\documentclass[preprint, 10p]{elsarticle}
\usepackage[utf8]{inputenc}
\usepackage[left=2cm,right=2cm,top=2cm,bottom=2cm]{geometry}

\usepackage{lineno,hyperref}
\modulolinenumbers[5]

\usepackage{amsmath}
\usepackage{amsthm}
\usepackage{amsfonts}
\usepackage{amssymb}
\usepackage{array}
\usepackage{caption}
\usepackage{enumitem}
\usepackage{framed}
\usepackage{graphicx}
\usepackage{pdfpages}
\usepackage{subcaption}
\usepackage{wrapfig}
\usepackage{xcolor}

\newtheorem{thm}{Theorem}
\newtheorem{lem}[thm]{Lemma}

\newdefinition{rmk}[thm]{Remark}
\newdefinition{defn}[thm]{Definition}
\newdefinition{exmp}[thm]{Example}

\journal{Applied Mathematical Letters}

\newcommand{\abs}[1]{\left\vert#1\right\vert}			% vector norm
\newcommand{\gu}{\mathrm{u}}

\newcommand{\gw}{\mathrm{w}}
\newcommand{\gz}{\mathrm{z}}

\def\mod{\mathop\mathrm{mod}\nolimits}
\newcommand{\Real}{\mathbb{R}}

\usepackage{multirow}

\usepackage{natbib}
\newcommand{\tabEntry}[2]{$#1(#2)$}

\usepackage{etoolbox}
\apptocmd\normalsize{%
 \abovedisplayskip=4pt
 \abovedisplayshortskip=3pt
 \belowdisplayskip=4pt
 \belowdisplayshortskip=6pt
}{}{}

\usepackage{enumitem}
\setlist{nosep}

\begin{document}

\begin{frontmatter}
\title{Counting and ordering periodic stationary solutions of lattice Nagumo equations}

\author[1]{Hermen Jan Hupkes}
\ead{hhupkes@math.leidenuniv.nl}

\author[1]{Leonardo Morelli}
\ead{l.morelli@math.leidenuniv.nl}

\author[2]{Petr Stehl\'{i}k\corref{cor1}}
\ead{pstehlik@kma.zcu.cz}

\author[2]{Vladim\'{i}r \v{S}v\'{i}gler}
\ead{sviglerv@kma.zcu.cz}

\cortext[cor1]{Corresponding author.}

\address[1]{Mathematisch Instituut, Universiteit Leiden, P.O. Box 9512, 2300 RA Leiden, The Netherlands}
\address[2]{Department of Mathematics and NTIS, University of West Bohemia, Univerzitn\'{i} 8, 30100 Pilsen, Czech Republic}

\begin{abstract}
We study the rich structure of periodic stationary solutions of Nagumo reaction diffusion equation on lattices. By exploring 
%its 
the relationship with Nagumo equations on cyclic graphs we are able to 
%prove the existence of equivalence classes of periodic stationary solutions which can be partially ordered and counted using the connection  with combinatorial words, necklaces, bracelets and Lyndon words. 
divide these periodic solutions into equivalence classes
that can be partially ordered and counted. In order to accomplish this,
we use combinatorial
concepts such as necklaces, bracelets and Lyndon words.
\end{abstract}

\begin{keyword}
%% keywords here, in the form: keyword \sep keyword
reaction diffusion equation \sep lattice differential equation \sep graph differential equations \sep periodic solutions \sep travelling waves \sep necklaces
%% PACS codes here, in the form: \PACS code \sep code

%% MSC codes here, in the form: \MSC code \sep code
%% or \MSC[2008] code \sep code (2000 is the default)
\MSC 34A33 \sep 37L60 \sep 39A12 \sep 65M22

% 34A33  	Lattice differential equations
% 34K31  	Lattice functional-differential equations
% 37L60  	Lattice dynamics
% 39A12     discrete version of topics in analysis
% 65M22      Solution of discretized equations
\end{keyword}

\end{frontmatter}

\section{Introduction}
\noindent In this paper we explore the structure of periodic stationary solutions of the lattice Nagumo equation
\begin{align}
\dot{u}_i(t) &= d(u_{i-1}(t) - 2 u_i(t) + u_{i+1}(t))+g(u_i(t);a),\quad i\in\mathbb{Z}, \quad t\in\mathbb{R}. \label{eq:lde} \tag{LDE}
\end{align}
We assume $d>0$ and consider the cubic bistable nonlinearity $g(u;a):=u(1-u)(u-a)$, $a\in(0,1)$. This equation has been extensively studied as the simplest model describing the competition between two stable states $u=0$ and $u=1$ in a spatially discrete environment. One of its key features is the existence of nondecreasing travelling waves $u_j(t) = \Phi(j- ct)$, and the fact that these waves do not move ($c=0$) for small values of $d$. This phenomenon (called pinning) is caused by the existence of heterogeneous stationary solutions, which prevent the dominance of the two stable homogeneous states $u=0$ and $u=1$. Our goal is to show that the periodic stationary solutions of \eqref{eq:lde}, which exist mainly inside the pinning region, form equivalence classes that can be partially ordered and counted. 

Equation \eqref{eq:lde} is a discrete-space version of the famous Nagumo reaction-diffusion PDE  
$u_t = d u_{xx} + g (u; a)$, with $x \in \mathbb{R}.$ The lattice counterpart \eqref{eq:lde} has a richer set of equilibria \cite{VL30} which in turn implies more complex behaviour of travelling and standing front solutions \cite{MPB, VL50}. Mallet-Paret \cite{MPB} established that for each $a \in [0,1]$ and $d > 0$ there exists a unique $c=c(a,d)$ for which the wave $\Phi$ exists. However, if we fix $a \in (0,1) \setminus \{\frac{1}{2} \}$, Zinner \cite{VL50} showed that $c(a , d) \neq 0$ for $d \gg 1$ and
 Keener \cite{VL28} proved that $c(a, d)= 0$ for $0 < d \ll 1$. Moreover, for fixed $d>0$ the results in \cite{HOFFMPcrys} suggest the existence of $\delta(d)>0$ so that $c(a ,d) = 0$ whenever $\abs{a - \frac{1}{2}} \le \delta(d)$. This above mentioned pinning is typical for lattice equations \cite{EVV2005AppMath, Guo2011, HJHVL2005}. Since the pinning region is dominated by heterogeneous (periodic and aperiodic) stationary solutions, our paper contributes to the understanding of this important feature (see Fig.~\ref{f:regions} for a simple illustration).

A second important motivation 
for understanding the
%investigating the structure of 
periodic stationary solutions of \eqref{eq:lde} is that this knowledge
aids us in the search for
so-called multichromatic waves.
%Besides the fact that our analysis contributes to the better understanding of the pinning phenomenon, the knowledge of structure of periodic stationary solutions of \eqref{eq:lde} enables to seek for the so-called multichromatic waves. 
These non-monotone traveling waves connect two or more $n$-periodic stationary solutions of \eqref{eq:lde} (in contrast to standard monochromatic waves which are monotone). In our companion papers \cite{bichrom, mchrom} we have shown that these waves exist mainly inside the pinning region, 
appearing and disappearing as $d$ increases. 
%Some of 
The waves that exist outside of the pinning region
can be combined to form
%which %in turn 
%leads to 
complex collision waves that involve direction changes.

In this paper we name, partially order and count the equivalence classes of  periodic stationary solutions of \eqref{eq:lde} based on the  connection between \eqref{eq:lde} and the Nagumo  equation posed on cyclic graphs. For an arbitrary undirected graph $\mathcal{G} = (V,E)$ with the set of vertices $V = \{ 1,2,3, \ldots, n \}$ and a set of edges $E$, the Nagumo equation on a graph $\mathcal{G}$ is\footnote{We use italic letters for double sequences (e.g., $u$ for solutions of \eqref{eq:lde}) and roman ones for vectors (e.g., $\gu$ for solutions of \eqref{eq:gde}).}
\begin{align}
\dot{\gu}_i(t) & = d \sum_{j \in \mathcal{N}(i)} \big(\gu_j(t)-\gu_i(t)\big) + g(\gu_i(t);a),\quad i\in V,t\in\mathbb{R},
\label{eq:gde} \tag{GDE}
\end{align}
where $\mathcal{N}(i)$ denotes the 1-neighbourhood of vertex $i \in V$.

In $\S$\ref{sec:tw-2} we establish the connection between the stationary solutions of  \eqref{eq:gde} and the periodic stationary solutions of \eqref{eq:lde}. In $\S$\ref{sec:naming} we use this connection and the implicit function theorem to build a naming scheme for periodic stationary solutions of \eqref{eq:lde}. In  $\S$\ref{sec:symmetries} we discuss their symmetries, which allows us to define and count their equivalence classes in $\S$\ref{sec:counting}.
This is achieved by establishing a link with
combinatorial concepts such as necklaces, bracelets and Lyndon words.
Our main result is formulated in Theorem \ref{t:counting}
and illustrated by simple examples.

%Using the naming scheme, we establish a link to the combinatorial necklaces, bracelets and Lyndon words. We conclude with our main result (Theorem \ref{t:counting}) which counts the number of various equivalence classes of periodic stationary solutions of \eqref{eq:lde} and is illustrated by simple examples.

\begin{center}
\begin{figure}
\centering
%\begin{minipage}{.33\textwidth}
%    \centering
%    \includegraphics[width=\textwidth]{./fig/pic_lyndon_regions_4.pdf}
%    \caption{The caption of monochromatic region is put in just for graphical reference.}
%    %\label{f:solutions}
%\end{minipage}
\begin{minipage}{.4\textwidth}
    \centering
    \includegraphics[width=\textwidth]{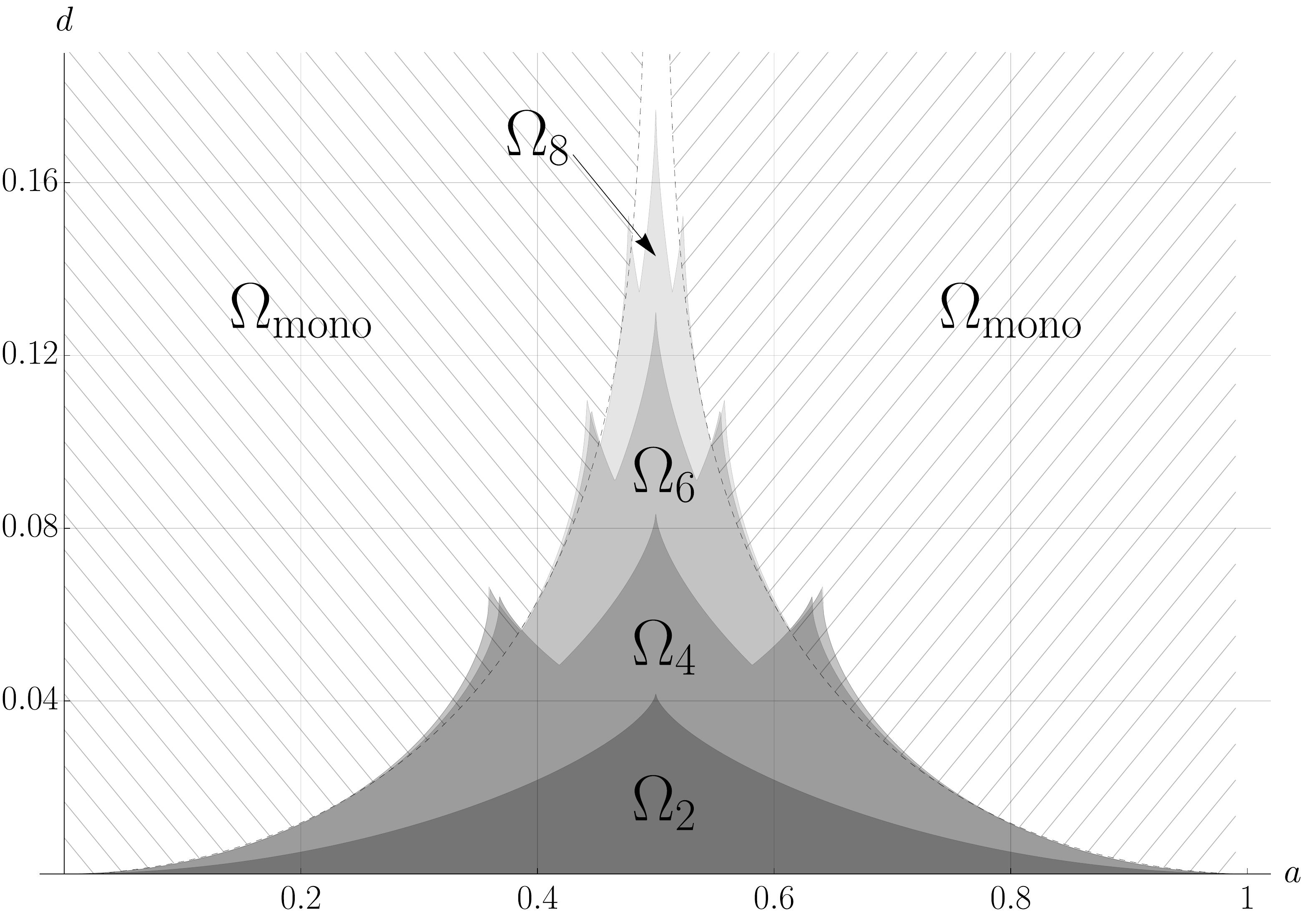}
    %\caption{The caption of monochromatic region is put in just for graphical reference. }
    %\label{f:solutions}
\end{minipage}
\begin{minipage}{.4\textwidth}
    \centering
    \includegraphics[width=\textwidth]{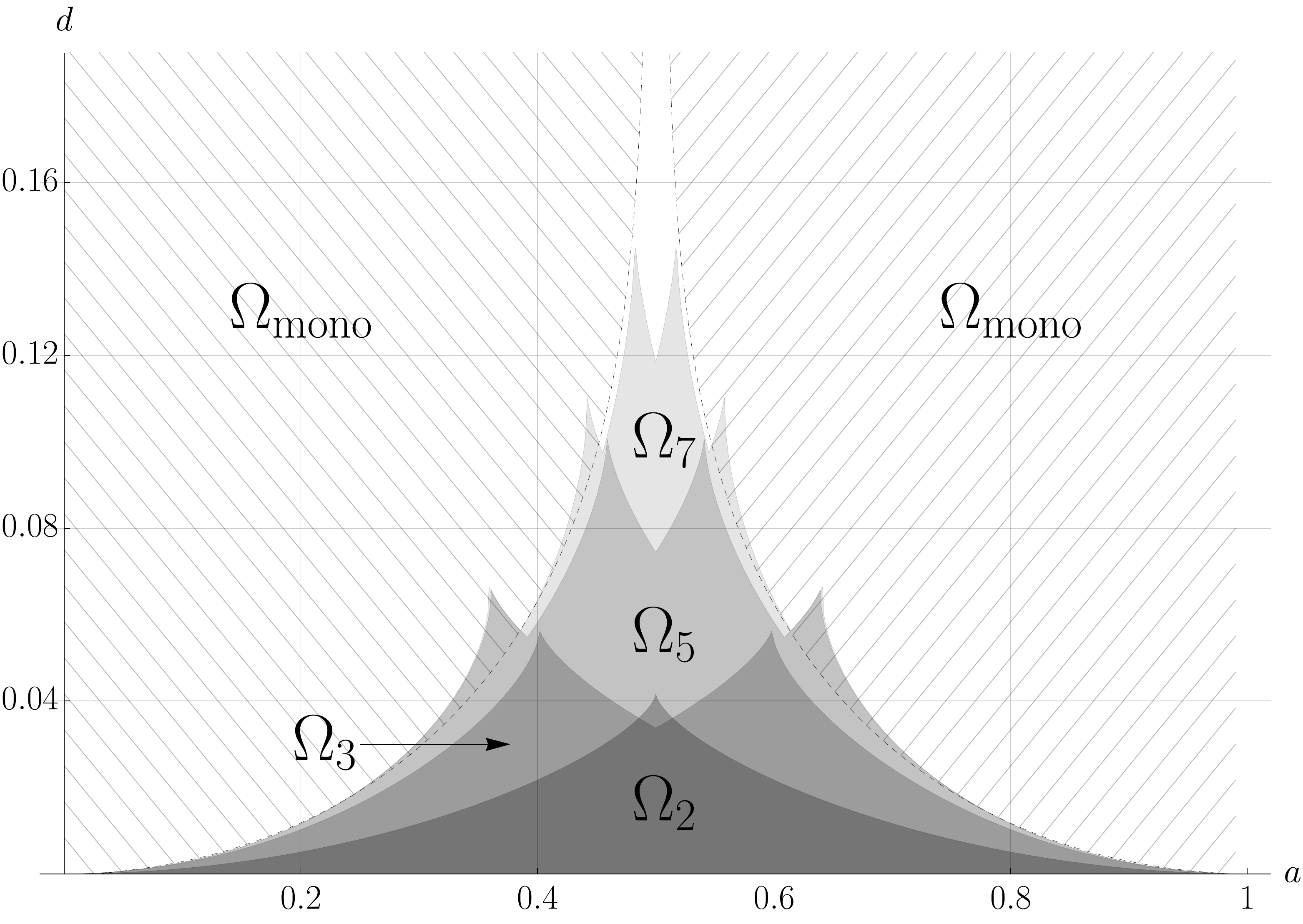}
    %\caption{The caption of monochromatic region is put in just for graphical reference. }
    %\label{f:solutions}
\end{minipage}
\caption{Existence regions $\Omega_n$ of stable $n$-perioidic stationary solutions of \eqref{eq:lde}. The left panel depicts even periods, the right one odd periods. The hatched regions $\Omega_{\text{mono}}$ correspond to parameter values for which the monochromatic waves of \eqref{eq:lde} travel.}\label{f:regions}
\end{figure}
\end{center}

%----------------------------------------------------------------------------------------
%----------------------------------------------------------------------------------------
%----------------------------------------------------------------------------------------

\section{Periodic solutions and solutions of graph Nagumo equation} \label{sec:tw-2}
%\subsection{Periodic Stationary Solutions and Stationary Solutions on $\mathcal{C}_n$}
\noindent We here consider
$\mathcal{G} = \mathcal{C}_n$, where $\mathcal{C}_n$ is a cycle graph on $n$ vertices.
The equation~\eqref{eq:gde} can now be written as $\dot{\gu}(t) = G(\gu(t);a,d)$, where $G:\mathbb{R}^n\rightarrow\mathbb{R}^n$ is given by
\begin{equation}\label{eq:G}
G(\gu;a,d) := \begin{pmatrix}
d (\gu_n-2\gu_1+\gu_2) + g\big(\gu_1; a\big) \\
d (\gu_1-2\gu_2+\gu_3) + g\big(\gu_2; a\big) \\
\vdots \\
d (\gu_{n-1}-2\gu_n+\gu_1) + g\big(\gu_n; a\big)
\end{pmatrix} \, .
\end{equation}
%[hjh: ? - add something like:, which reduces to ... for $n=2$]
Our key results are based on the correspondence of stationary solutions of \eqref{eq:gde} on $\mathcal{G}=\mathcal{C}_n$ and periodic stationary solutions of \eqref{eq:lde}\footnote{Additionally, it is well-known that the problem of finding stationary solutions of graph differential equations on cycles $\mathcal{G}=\mathcal{C}_n$ is actually equivalent to periodic discrete boundary value problems \cite{Stehlik2017a, Volek2016}.}. We say that a double sequence $u=(u_i)_{i\in\mathbb{Z}}$ is a periodic extension of a vector $\gu\in\mathbb{R}^n$ if $u_i = \gu_{\mod(i,n)}$ (we assume that the modulo operator takes values    $\mod(i,n) \in \{1, \ldots , n \}$).
We remark that \eqref{eq:lde} is well-posed
as an evolution equation on the space $\ell^\infty(\mathbb{Z};\Real)$.
However, we caution
the reader that lattice equations do not necessarily have unique solutions
if one drops this boundedness condition,
even in the linear case \cite{Chen2008, Slavik2015}.
%lattice equations do not have unique solutions 

% [hjh: I think we can just say $n \ge 2$ here.]
\begin{lem}\label{l:tw-equivalence}
Let $\mathcal G=\mathcal C_n$, $n\geq 3$, be a cycle graph on $n$ vertices. 
The vector $\gu=(\gu_1,\gu_2,\ldots,\gu_n)$ is a stationary solution of \eqref{eq:gde} on $\mathcal G=\mathcal C_n$ if and only if its periodic extension $u$ is an $n$-periodic stationary solution of \eqref{eq:lde}. Moreover, $\gu$ is an asymptotically stable solution of \eqref{eq:gde} if and only if $u$ is an asymptotically stable solution of \eqref{eq:lde} with respect to the $\ell^\infty$-norm.
\end{lem}

%\begin{thm}\label{thm:tw-equivalence}
%Let $\mathcal G=\mathcal C_n$, $n\geq 3$, be a cycle graph on $n$ vertices, $t_0 \in \mathbb{R}$, $\mathrm{v}_0 \in \mathbb{R}^n$ be given and $u_0$ be its periodic extension. Let $\gv(t)$ be the unique solution of GDE~\eqref{eq:gde} on $\mathcal{G} = \mathcal{C}_n$ with $\gv(t_0)=\gv_0$ and let $u(t) = (u_i(t))_{i \in \mathbb{Z}}$ be the unique bounded solution of LDE~\eqref{eq:lde} with $u(t_0)=u_0$.
%
%Then $u(t)$ is a periodic extension of $\gv(t)$ for all $t \in [t_0, +\infty)$.
%\end{thm}

\begin{proof}
A short inspection readily yields the desired equivalence between solutions
of \eqref{eq:lde} and \eqref{eq:gde}\footnote{We omit the case of $n=2$. In this case, a slightly modified version of Lem.~\ref{l:tw-equivalence} holds. The reduced version of \eqref{eq:G} for $n=2$
\[
G(\gu;a,d) := \begin{pmatrix}
d (\gu_2-2\gu_1+\gu_2) + g\big(\gu_1; a\big) \\
d (\gu_1-2\gu_2+\gu_1) + g\big(\gu_2; a\big) \\
\end{pmatrix} = 
\begin{pmatrix} 
2d (\gu_2-\gu_1) + g\big(\gu_1; a\big) \\
2d (\gu_1-\gu_2) + g\big(\gu_2; a\big) \\
\end{pmatrix}
\, .
\]
implies that solutions of \eqref{eq:lde} and \eqref{eq:gde} are equivalent if one considers the double value of $d$ in \eqref{eq:lde}.
}. Turning to their stability,
%
%If $\gu=(\gu_1,\ldots,\gu_n)$ is a stationary solution of\eqref{eq:gde} on $\mathcal G=\mathcal C_n$ then $0=G(\gu;a,d)$. If $u=(u_i)_{i\in\mathbb{Z}}$ is a stationary periodic solution of \eqref{eq:lde} then it satisfies the infinite system of equations
%\begin{equation}\label{eq:proof:LDE:expanded}
%0 =d (u_{i-1}-2u_i+u_{i+1}) + g\big(u_i; a\big),\quad i\in\mathbb{Z}.
%\end{equation}
%Apparently, if $(u_i)_{i\in\mathbb{Z}}$ is a stationary $n$-periodic solution of \eqref{eq:lde}, we have $u_{-1}=u_{n-1}$, $u_0=u_n$ etc. Consequently, for $i=1$ and $i=n$, the equation~\eqref{eq:proof:LDE:expanded} becomes
%\[
%0 =d (u_n-2u_1+u_2) + g\big(u_1; a\big), \quad 0 =d (u_{n-1}-2u_n+u_1) + g\big(u_n; a\big).
%\]
%This implies immediately that $(\gu_1,\gu_2,\ldots,\gu_n)$ is a stationary solution of \eqref{eq:gde} if and only if 
%\begin{align*}
%u &= (\ldots,u_0,u_1,\ldots,u_{n-1},u_n,u_{n+1},\ldots,u_{2n},\ldots)= (\ldots,\gu_n,\gu_1,\ldots,\gu_{n-1},\gu_n,\gu_1,\ldots,\gu_{n},\ldots)
%\end{align*}
%is a stationary $n$-periodic solution of \eqref{eq:lde}.
%
let us assume that  $\gu^*=(\gu_1^*,\gu_2^*,\ldots,\gu_n^*)$ is an asymptotically stable solution of \eqref{eq:gde}. 
There hence exists
%This implies that there exists 
$\gamma>0$ such that for each $\gu_0\in\mathbb{R}^n$ with $\|\gu_0-\gu^* \|<\gamma$ we have
\[
\lim_{t\rightarrow\infty} \gu(t,\gu_0) = \gu^*,
\]
in which $\gu(t,\gu_0)$ denotes the solution of \eqref{eq:gde} 
with the initial condition $\gu_0$. Consequently, there exists $\delta>0$ so that the vectors $\gw_0, \gz_0 \in\mathbb{R}^n$ 
defined by
\[
(\gw_0)_i = \gu^*_i+\delta, \quad (\gz_0)_i=\gu^*_i-\delta \quad \text{for all } i=1,2,\ldots,n,
\]
satisfy $
\lim_{t\rightarrow\infty} \gw(t,\gw_0) = \gu^*,\quad \lim_{t\rightarrow\infty} \gz(t,\gz_0) = \gu^*.
$

Let us now consider the periodic extensions $u^*$, $w_0$ and $z_0$ of 
the vectors $\gu^*$, $\gw_0$ and $\gz_0$. Then the 
corresponding %unique bounded 
solutions $w(t,w_0)$, $z(t,w_0)$
of \eqref{eq:lde} satisfy
%denoted by $w(t,w_0)$, $z(t,w_0)$ satisfy for each $t$
\[
w_i(t,w_0)=\gw_{\mod(i,n)}(t,\gw_0),\quad z_i(t,w_0)=\gz_{\mod(i,n)}(t,\gz_0),
\]
for each $t \ge 0$, which implies
%Consequently,
\[
\lim_{t\rightarrow\infty} w(t,w_0) = u^*,\quad \lim_{t\rightarrow\infty} z(t,z_0) = u^*.
\]
Using the comparison principle (e.g., Chen et al. \cite[Lemma 1]{Chen2008}) we can hence conclude %see 
that all %bounded 
solutions $u$ of \eqref{eq:lde} with an initial condition $u_0$ 
that satisfies %satisfying 
$\| u_0-u^* \|_\infty < \delta $ 
indeed have % admit the limit
$ \lim_{t\rightarrow\infty} u(t,u_0) = u^*,$ since
\[
u^* \leftarrow z(t,z_0) \leq u(t,u_0) \leq w(t,w_0) \rightarrow u^*.
\]
The opposite implication can be proved similarly.
\end{proof}

%\begin{rmk}[hjh: this whole remark can be absorbed
%in the start of this section?]
%Lemma~\ref{l:tw-equivalence} can be extended to the case of $2$-periodic %stationary solutions. 
%The only difference is that the double value of the diffusion parameter $d$ %must be considered in~\eqref{eq:G}, since:
%\begin{equation}\label{eq:G2}
%G(\gu;a,d) := \begin{pmatrix}
%d (\gu_2-2\gu_1+\gu_2) + g\big(\gu_1; a\big) \\
%d (\gu_1-2\gu_2+\gu_1) + g\big(\gu_2; a\big) \\
%\end{pmatrix} = 
%\begin{pmatrix} 
%2d (\gu_2-\gu_1) + g\big(\gu_1; a\big) \\
%2d (\gu_1-\gu_2) + g\big(\gu_2; a\big) \\
%\end{pmatrix}
%\, .
%\end{equation}
%
%[hjh: this has been moved to start of section]
%Moreover, let us emphasize that we only consider bounded solutions of %\eqref{eq:lde} and $\ell^\infty$-norm since in general initial problems for %lattice equations do not have unique solutions even in the linear case %(\cite{Chen2008, Slavik2015}). Additionally, let  us note that the problem of %finding stationary solutions of graph differential equations on cycles is %actually equivalent to periodic discrete boundary value problems %\cite{Stehlik2017a, Volek2016}.
%\end{rmk}

%----------------------------------------------------------------------------------------
%----------------------------------------------------------------------------------------
%----------------------------------------------------------------------------------------

\section{Naming scheme for stationary periodic solutions}\label{sec:naming}
\noindent The equivalence between $n$-periodic solutions of \eqref{eq:lde} and solutions of \eqref{eq:gde} on $\mathcal{G}=\mathcal{C}_n$ (see Lem.~\ref{l:tw-equivalence}) allows us to focus on the latter
%graph Nagumo equation 
in order to establish our naming scheme 
for the former solutions.
%$n$-periodic stationary solutions of \eqref{eq:lde}. 
First, let us observe that $G(\gu; a, 0) = 0$ for any $a \in (0,1)$ and $\gu \in \{0, a, 1\}^n$. Moreover, the fact that
\begin{equation}
\label{eq:mcr:diag:matrix:d1g}
    D_1 G (\gu  ; a, 0) = \mathrm{diag}\Big( g'( \gu_1; a \big), \ldots , g'\big(\gu_n; a \big) \Big)
\end{equation}
has non-zero entries allows us to employ the implicit function theorem 
to conclude 
%and get 
that there are $3^n$ solution branches emanating out of the 
roots $\{0, a, 1\}^n$ for $d$ small. These branches can be tracked up until the first collision with another branch. This justifies the use of the following naming scheme for $n$-periodic stationary solutions.

We introduce an alphabet $A=\{\mathfrak{0} , \mathfrak{a}, \mathfrak{1} \}$ and call $\gu_\gw\in [0,1]^n$ a stationary solution  of type $\gw \in A^n = \{\mathfrak{0} , \mathfrak{a}, \mathfrak{1} \}^n$ if it satisfies
$G(\gu_\gw; a, d) = 0$ and lies on the branch
emanating from the root %[hjh: why not use $\gw_a$ for
%consistency with mchrom paper] %$\{ 0, a, 1 \}^n$ at $d = 0$.
%which can be connected to 
%[hjh: just write $\gw_a$ here without the inclusion?]
$\gw_a$ at $d=0$,
where $\gw_a: \{\mathfrak{0} , \mathfrak{a}, \mathfrak{1} \}^n \rightarrow \{0, a, 1\}^n$ is defined by
\[
\left(\gw_a \right)_i = \left\{ \begin{array}{lcl} 
                             0 & & \hbox{if } \gw_i = \mathfrak{0}, \\
                             a & & \hbox{if } \gw_i = \mathfrak{a},  \\
                             1 & & \hbox{if } \gw_i = \mathfrak{1} .
                         \end{array}
                        \right.
\]

%\begin{defn}\label{d:mcr:eq:type}
%Consider a word $\gw \in \{\mathfrak{0} , \mathfrak{a}, \mathfrak{1} \}^n$, a triplet $ (\gu, a,d) \in [0,1]^n %\times (0, 1) \times [0, \infty).$ Then we say that $\gu$ is an equilibrium
%of type $\gw$ if there exists a $C^1$-smooth curve
%\begin{equation}
%\label{eq:gamma:curve}
%[0,1] \ni t \mapsto \big( \gv(t),  \alpha(t), \delta(t)  \big) \in [0,1]^n \times (0,1) \times [0, \infty)
%%\Gamma_\gw := \{ (\gv(t),  \alpha(t), \delta(t) ), t\in[0,1]: \alpha(0) = \alpha(1) = a, \delta(0) = 0, %\delta(1)=d, \gv(0)=\gw_a \text{ and } %\gv(1)=\gu \}
%\end{equation}
%so that we have
%\begin{equation}
%\label{eq:mc:path:reqs:basic}
%  (\gv, \alpha,\delta)(0)  =  (\mathcal{N}(\gw), a , 0), \quad   (\gv, \alpha,\delta)(1)  =  (\gu, a , d),
%\end{equation}
%and $G\big( \gv(t) ; \alpha(t), \delta(t) \big) = 0$, and $\det D_1 G\big( \gv(t); \alpha(t), \delta(t) \big) \neq %0$ hold for all $0 \le t \le 1$.
%\end{defn}
Using this definition we introduce connected sets 
\begin{equation}
\Omega_{\gw} = \{ ( a, d) \in \mathcal{H} : \hbox{the system } G
  (\cdot \, ;a, d) = 0 \hbox{ admits an equilibrium of type } \gw \},\label{eq:Omega}
\end{equation}
which are open in the half-strip $\mathcal{H} =  [0,1] \times [0, \infty)$. While a full analysis of the sets $\Omega_{\gw}$ can be very tricky (for a fixed $a\in(0,1)$, solutions of type $\gw$ can disappear and then reappear, see \cite{mchrom}), we are only interested in small values of $d$ here in this paper.
%(see Fig.~\ref{f:regions:4} for illustration of regions in the case of $n=4$).

\begin{lem}\label{l:3n2n}
Let $a\in(0,1)$ and $d>0$ be small enough. Then \eqref{eq:lde} has $3^n$ stationary $n$-periodic solutions. %out of which solutions 
These solutions are asymptotically stable
if and only if they belong to the $2^n$ solutions 
of type $\gw \in \{\mathfrak{0},\mathfrak{1}\}^n$. %form $2^n$ 
%are all asymptotically stable. % stationary $n$-periodic solutions.
\end{lem}
\begin{proof}
The existence of $3^n$ stationary $n$-periodic solutions 
follows from Lemma \ref{l:tw-equivalence} 
and the $3^n$ solution branches for  \eqref{eq:gde} 
supplied by the implicit function theorem.
%and . 
The % fact that these solutions are asymptotically stable if and only if they are of type $\gw\in \{\mathfrak{0},\mathfrak{1}\}^2$, follows from
stability properties follow from 
\eqref{eq:mcr:diag:matrix:d1g} and the fact that $g'(0;a)=-a<0$, $g'(1;a)=a-1<0$ and $g'(a;a)=a(1-a)>0$.
\end{proof}

\noindent %[hjh:todo introduce $u_{\gw}$ notation] 
Moreover, pairs of stable solutions can be ordered if the corresponding words are ordered.

\begin{lem}\label{l:ordering}
Let $a\in(0,1)$ and $d>0$ be small enough and consider a distinct pair $\gw_A,\gw_B \in \{\mathfrak{0},\mathfrak{a},\mathfrak{1}\}^n$ with $(\gw_A)_i \le (\gw_B)_i$ for all $i$. Suppose furthermore that at least one of these two words is contained in $\{\mathfrak{0}, \mathfrak{1} \}^n$.
Then the solutions $u_{\gw_A}, u_{\gw_B}$ of \eqref{eq:lde} satisfy the strict component-wise inequality $(u_{\gw_A})_i < (u_{\gw_B})_i$, for all $i\in\mathbb{Z}$.
\end{lem}
\begin{proof}
The proof follows from \cite[Lemma 5.2]{mchrom} and Lemma \ref{l:tw-equivalence}.
\end{proof}

\begin{figure}
\centering
\begin{minipage}{.4\textwidth}
    \centering
    \includegraphics[width=\textwidth]{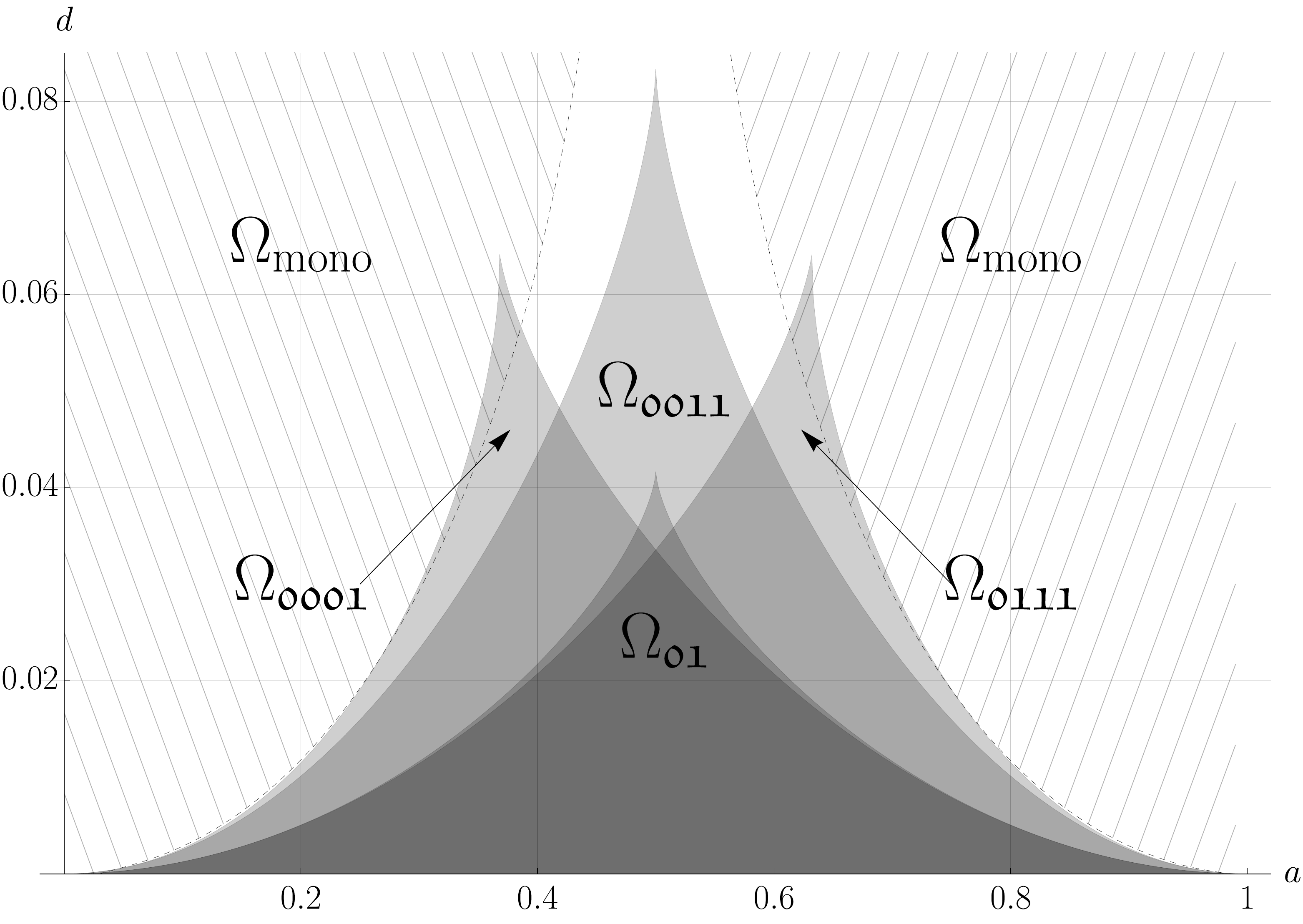}
    \caption{Regions $\Omega_\gw$ defined in \eqref{eq:Omega} corresponding to asymptotically stable spatially heterogeneous $4$-periodic solutions of \eqref{eq:lde}: $u_\mathfrak{0001}, u_\mathfrak{0011}, u_\mathfrak{01}, u_\mathfrak{0111}$. The hatched regions $\Omega_{\text{mono}}$ correspond to pairs $(a,d)$ for which the monochromatic waves of \eqref{eq:lde} travel.}
    \label{f:regions:4}
\end{minipage}\quad
\begin{minipage}{.4\textwidth}
    \centering
    \includegraphics[width=\textwidth]{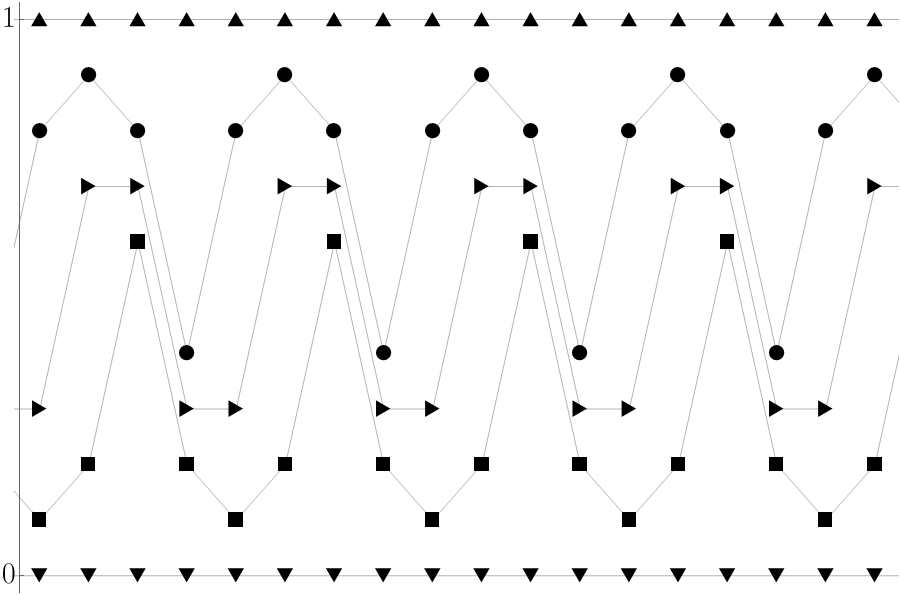}
    \caption{Lyndon representatives of 5 ordered classes of asymptotically stable stationary $4$-periodic solutions of \eqref{eq:lde}: \protect\includegraphics[width=2mm]{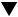} - $u_\mathfrak{0}$, \protect\includegraphics[width=2mm]{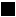} - $u_\mathfrak{0001}$,
    \protect\includegraphics[width=2mm]{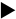} - $u_\mathfrak{0011}$,
    \protect\includegraphics[width=2mm]{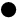} - $u_\mathfrak{0111}$,
    \protect\includegraphics[width=2mm]{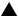} - $u_\mathfrak{1}$ (the values are slightly modified for better visualisation). 
    }
    \label{f:4periodic}
\end{minipage}
\end{figure}

%----------------------------------------------------------------------------------------
%----------------------------------------------------------------------------------------
%----------------------------------------------------------------------------------------

\section{Symmetries of stationary periodic solutions}\label{sec:symmetries}
\noindent The naming scheme introduced above allows us to study two key symmetries among the $3^n$ stationary solutions of  \eqref{eq:gde} and the corresponding $n$-periodic stationary solutions of \eqref{eq:lde}.
\paragraph{Translation (rotation)} If $(u_i)$ is an $n$-periodic stationary solution of \eqref{eq:lde} then this also holds for $(u_{i+k})$. 
%[hjh: later on we use $\mathcal{T}_i$ because $k$ is reserved
%for something else. How about we unify to $\mathcal{T}_{\ell}$
%everywhere?] 
We define the translation (rotation) operator on words (and more generally on any vectors of length $n$) by
\[
\left(\mathcal{T}_\ell \gw\right)_i := \gw_{\mod(i+\ell,n)}.
\]
\paragraph{Reflection} If $(u_i)$ is an $n$-periodic stationary solution of \eqref{eq:lde} then the same is true for  $(u_{1-i})$. We define the reflection operator by
\[
\left(\mathcal{R} \gw\right)_i := \gw_{\mod(1-i,n)}.
\]

Trivially, if $\gu$ is a solution of type $\gw$ 
%of GDE
for $G(\gu; a, d) = 0 $ then %[hjh: do we want brackets or not?]
$\mathcal{T}_\ell\gu$ is a solution of type $\mathcal{T}_\ell\gw$  and $\mathcal{R}\gu$ is a solution of type $\mathcal{R}\gw$, which immediately implies that
\[
\Omega_{\gw} = \Omega_{\mathcal{T}_1 \gw} =  \ldots = \Omega_{\mathcal{T}_{n-1} \gw } = \Omega_{\mathcal{R} \gw }.
\]

Naturally, other symmetries can be considered as well. For example, in the case of $a=1/2$, it makes sense to consider symbol swapping  $\mathfrak{0} \leftrightarrow \mathfrak{1}$. However, %in this case 
the existence of solution of type $\gw$ for a given pair $(a,d)$ does not imply the existence of solutions for the ``swapped'' word $\tilde{\gw}$ for general $a\neq 1/2$, see Fig.~\ref{f:regions}. Therefore, we only focus on the simplest symmetries - translations $\mathcal{T}_\ell$ and reflections $\mathcal{R}$.

%\paragraph{Permutation $\mathfrak{0} \leftrightarrow \mathfrak{1}$} The properties of third symmetry is a bit more complicated. Let the permutation $\mathfrak{0} \leftrightarrow \mathfrak{1}$ be defined by
%\[
%\left(\mathcal{P} \gw\right)_i :=  \left\{ \begin{array}{lcl} 
%                             \mathfrak{1} & &\hbox{if } \gw_i = \mathfrak{0}, \\
%                             \mathfrak{a} & & \hbox{if } \gw_i = \mathfrak{a},  \\
%                             \mathfrak{0} & &\hbox{if } \gw_i = \mathfrak{1}. \\
%                         \end{array}
%                        \right.
%\]
%If $\gu$ is a solution of type $\gw$ of GDE $G(\gu; a, d) = 0 $ then $1-\gu$ is a solution of type type $\mathcal{P}\gw$ of GDE $G(\gu; 1- a, d) = 0 $, since $g(1 - u; a) = -g(u; 1 - a)$. Consequently, $\Omega_{\gw} \neq \Omega_{\mathcal{P} \gw}$  but satisfies
%\[
%\Omega_{\mathcal{P} \gw} = \{(a,d): (1-a,d)\in \Omega_{\gw} \}.
%\]
%Therefore, if a solution of type $\gw$ exists for a given pair $(a,d)$, there need not exist solution of type $\mathcal{P}\gw$. Additionally, let us note that we do not mention permutations involving the letter $\mathfrak{a}$ because these would change the key qualitative properties of solutions, for example, stability (see Lemma \ref{l:3n2n}) and partial ordering (see Lemma \ref{l:ordering}).

%----------------------------------------------------------------------------------------
%----------------------------------------------------------------------------------------
%----------------------------------------------------------------------------------------

\section{Counting equivalence classes of stationary periodic solutions}\label{sec:counting}
\noindent We can now define equivalence classes of $n$-periodic solutions to \eqref{eq:lde} by factoring out one or both of the symmetries discussed above. If we consider translations $\mathcal{T}_\ell$ and the word $\gw=\mathfrak{00a1}$ we have the following equivalence class of stationary $4$-periodic solutions of \eqref{eq:lde}:
\[
 \left[ u_\mathfrak{00a1}\right]_\mathcal{T} = \{u_\mathfrak{00a1}, \mathcal{T}_1u_\mathfrak{00a1}, \mathcal{T}_2 u_\mathfrak{00a1}, \mathcal{T}_3 u_\mathfrak{00a1} \}=\{u_\mathfrak{00a1}, u_\mathfrak{0a10}, u_\mathfrak{a100}, u_\mathfrak{100a} \}.
\]
If we consider both translations $\mathcal{T}_\ell$ and reflections $\mathcal{R}$, we have for example $\mathcal{R}u_\mathfrak{00a1} =  u_\mathfrak{1a00}$ and thus
\[
 \left[ u_\mathfrak{00a1}\right]_\mathcal{TR} = \{u_\mathfrak{00a1}, u_\mathfrak{0a10}, u_\mathfrak{a100}, u_\mathfrak{100a}, u_\mathfrak{1a00}, u_\mathfrak{01a0}, u_\mathfrak{001a}, u_\mathfrak{a001} \}.
\]
%\noindent We can now define equivalence classes of $n$-periodic solutions to \eqref{eq:lde} by factoring out one or both of the symmetries discussed above. For example, if we consider translations $\mathcal{T}_\ell$ and $\gw=\mathfrak{0001}$ we have the following equivalence classes of words, stationary solutions of \eqref{eq:gde} and stationary $n$-periodic solutions of \eqref{eq:lde}:
%\noindent [hjh: something like 'We can now define
%equivalence classes of $n$-periodic solutions to \eqref{eq:lde}
%by factoring out one or both of the symmetries discussed above.]
%Taking 
%%each 
%one or both of these symmetries we are able to define the equivalence classes of words, corresponding to stationary solutions of \eqref{eq:gde} and stationary $n$-periodic solutions of \eqref{eq:lde}. For example, if we consider translations $\mathcal{T}_i$ and $\gw=\mathfrak{0001}$ we have the following equivalence classes: 
%[hjh: I'm still not a fan of this line. In the sequel we only discuss
%the $n$-periodic solutions, so why bother with this
%distinction here? I would really just give an example with $\mathcal{T}$
%and an example with $\mathcal{T}\mathcal{R}$ just to illustrate the concept -
%and only for the $n$-periodics]
%\[
%[\mathfrak{0001}]_\mathcal{T} = \{\mathfrak{0001}, \mathfrak{0010}, \mathfrak{0100}, \mathfrak{1000} \}, \ \left[ \gu_\mathfrak{0001}\right]_\mathcal{T} = \{\gu_\mathfrak{0001}, \gu_\mathfrak{0010}, \gu_\mathfrak{0100}, \gu_\mathfrak{1000} \}, \ \left[ u_\mathfrak{0001}\right]_\mathcal{T} = \{u_\mathfrak{0001}, u_\mathfrak{0010}, u_\mathfrak{0100}, u_\mathfrak{1000} \}.
%\]
We always use the smallest word in the lexicographical sense (the so called Lyndon word) as a class representative.
%and 
If the word is not primitive (i.e., it is periodic itself), we take the primitive (aperiodic) subword, e.g., $[\mathfrak{01}]_\mathcal{T}= [\mathfrak{0101}]_\mathcal{T}$.

If one is simply interested in the periodic solutions themselves then it is reasonable to factor out both translation and reflection symmetries and consider the full alphabet $\{\mathfrak{0}, \mathfrak{a}, \mathfrak{1} \}$.  However, in special circumstances (e.g., when connecting these solutions via multichromatic waves \cite{mchrom}) it only makes sense to factor out the translation symmetries
and to consider the reduced alphabet $\{\mathfrak{0}, \mathfrak{1} \}$.

%[hjh: formulate sharper. Say: if one is simply interested
%in the periodic solutions themselves then reasonable to 
%factor out both symmetries. However, when connecting
%these solutions via multichromatic waves it only makes
%sense to factor out the translation symmetries
%and to consider the reduced alphabet... 
%For reduction purposes it is also important
%to know whether an equivalence class is primitive or not.]
%We emphasize that in some circumstances it may be reasonable to include both symmetries, in other circumstances only translations may be reasonable (for example if one considers multichromatic waves travelling in a given direction \cite{mchrom}). Similarly, one may be interested both in equivalence classes related to the full alphabet $\{\mathfrak{0} , \mathfrak{a}, \mathfrak{1} \}$ or to the reduced alphabet $\{\mathfrak{0},  \mathfrak{1}\}$ if one considers only asymptotically stable stationary $n$-periodic solutions (see Lemma \ref{l:3n2n}). An equivalence class can be formed by solutions which are also periodic with a smaller period $m<n$ or by those for which $n$ is the smallest period. We talk about primitive $n$-periodic solutions in the latter case.

We can now use the combinatorial theory of words \cite{Lothaire1997, Sawada2001} to count the equivalence classes and to get the number of qualitatively different $n$-periodic stationary solutions. 
To this end, we 
%Let us 
define the %following 
quantities
\begin{align}
N_k(n)&=\frac{1}{n}\sum\limits_{d: d|n} \varphi(d) k^{\frac{n}{d}}, \quad L_k(n)=\frac{1}{n}\sum\limits_{d: d|n} \mu(d) k^{\frac{n}{d}},\label{eq:NknLkn}\\
B_k(n)&=\begin{cases}
\frac{1}{2} \left( N_k(n) + \frac{k+1}{2} k^{n/2}\right) & \text{for } n \text{ even},\\
\frac{1}{2} \left( N_k(n) + k^{(n+1)/2}\right) & \text{for } n \text{ odd},
\end{cases} \label{eq:Bkn}\\
BL_k(n) &= \sum\limits_{d: d|n} \mu(d) B_k(n/d), \label{eq:BLkn}
\end{align}
where $ \varphi(d)$ is the Euler's totient function and $\mu(d)$ is the M\"{o}bius function \cite[Chapter 2]{Apostol1998}. We can use these  quantities to formulate our main result, which describes the number of equivalence classes of $n$-periodic stationary solutions.

\begin{thm}\label{t:counting}
Pick an integer $n\geq 2$ and a parameter $a\in(0,1)$.
%, an integer $n\geq 2$. 
Then for $d>0$ small enough \eqref{eq:lde} has exactly
\begin{enumerate}
\item $3^n$ $n$-periodic stationary solutions which form
	\begin{enumerate}
	\item $N_3(n)$ equivalence classes with respect to translations $\mathcal{T}_\ell$. Moreover, $L_3(n)$ of these equivalence classes are formed by primitive periodic solutions.
	\item $B_3(n)$ equivalence classes with respect to translations $\mathcal{T}_\ell$ and reflections $\mathcal{R}$.  Moreover, $BL_3(n)$ of these equivalence classes are formed by primitive periodic solutions.	
	\end{enumerate}

\item $2^n$ asymptotically stable $n$-periodic stationary solutions which form
	\begin{enumerate}
	\item $N_2(n)$ equivalence classes with respect to translations $\mathcal{T}_\ell$. Moreover, $L_2(n)$ of these equivalence classes are formed by primitive periodic solutions.
	\item $B_2(n)$ equivalence classes with respect to translations $\mathcal{T}_\ell$ and reflections $\mathcal{R}$.  Moreover, $BL_2(n)$ of these equivalence classes are formed by primitive periodic solutions.	
	\end{enumerate}
\end{enumerate}
\end{thm}

\begin{proof}
A $k$-ary necklace of length $n$ is an equivalence class of words of length $n$ formed by $k$ letters which are equivalent with respect to translations (rotations) $\mathcal{T}_\ell$. There are $N_k(n)$ different necklaces \cite[Eq. (2.1)]{Sawada2001}. There are $L_k(n)$ distinct primitive (aperiodic) necklaces -- Lyndon words %[hjh: replace the - sign] 
\cite[Eq. (2.2)]{Sawada2001}. 

A $k$-ary bracelet of length $n$ is an equivalence class of words of length $n$ formed by $k$ letters which are equivalent with respect to translations (rotations) $\mathcal{T}_\ell$ and refection $\mathcal{R}$. There are $B_k(n)$ different bracelets \cite[Eq. (2.4)]{Sawada2001}. A primitive (aperiodic) bracelet is called a Lyndon bracelet. There are $BL_k(n)$ distinct Lyndon bracelets, which can be proved by the direct application of the M\"{o}bius inversion formula (e.g., \cite[Theorem 2.9]{Apostol1998}). The result for $n$-periodic stationary solutions of \eqref{eq:lde} then follows directly from Lemma \ref{l:3n2n}.
\end{proof}

Table \ref{tab:number:of:words} provides a summary of these results for small periods. As an example, we give a detailed
description of the equivalence classes for $n=3$ and $n=4$.
%and the following example illustrates in detail the case of $3$- and %$4$-periodic solutions.
%\begin{table}[t]
%\begin{center}
%\begin{tabular}{cccccc}
%\hline \hline
%&& \multicolumn{2}{c}{translation $\mathcal{T}$} & \multicolumn{2}{c}{translation $\mathcal{T}$+refection $\mathcal{R}$ } \\ \cline{3-4} \cline{5-6}
%Period & All solutions & All & Primitive & All & Primitive \\
%%$n$ & \tabEntry{3^n}{2^n} & \tabEntry{N_3(n)}{N_2(n)}  & \tabEntry{L_3(n)}{L_2(n)}  & \tabEntry{B_3(n)}{B_2(n)}  & \tabEntry{BL_3(n)}{BL_2(n)}  \\ \hline \hline
%$n$ & $k^n$ & $N_k(n)$  & $L_k(n)$  & $B_k(n)$  & $BL_k(n)$  \\ \hline \hline
%1 & \tabEntry{3}{2} & \tabEntry{3}{2} & \tabEntry{3}{2} & \tabEntry{3}{2} & \tabEntry{3}{2} \\
%2 & \tabEntry{9}{4} & \tabEntry{6}{3} & \tabEntry{3}{1} & \tabEntry{6}{3} & \tabEntry{3}{1} \\
%3 & \tabEntry{27}{8} & \tabEntry{11}{4} & \tabEntry{8}{2} & \tabEntry{10}{4} & \tabEntry{7}{2} \\
%4 & \tabEntry{81}{16} & \tabEntry{24}{6} & \tabEntry{18}{3} & \tabEntry{21}{6} & \tabEntry{15}{3} \\
%5 & \tabEntry{243}{32} & \tabEntry{51}{8} & \tabEntry{48}{6} & \tabEntry{39}{8} & \tabEntry{36}{6} \\
%6 & \tabEntry{729}{64} & \tabEntry{130}{14} & \tabEntry{116}{9} & \tabEntry{92}{13} & \tabEntry{79}{8} \\
%\ldots \\ \hline\hline
%\end{tabular}
%\end{center}
%\caption{Number of equivalence classes of $n$-periodic solutions of LDE \eqref{eq:lde}. The former number correspond to all equivalence classes, the latter to the equivalence classes formed by asymptotically stable solutions.}\label{tab:number:of:words}
%\end{table}

\begin{table}[t]
\begin{center}
\begin{tabular}{ccccccc}
\hline \hline
&& \multicolumn{2}{c}{translation $\mathcal{T}$} && \multicolumn{2}{c}{translation $\mathcal{T}$+refection $\mathcal{R}$ } \\ \cline{3-4} \cline{6-7}
Period & All solutions & All & Primitive && All & Primitive \\
%$n$ & \tabEntry{3^n}{2^n} & \tabEntry{N_3(n)}{N_2(n)}  & \tabEntry{L_3(n)}{L_2(n)}  & \tabEntry{B_3(n)}{B_2(n)}  & \tabEntry{BL_3(n)}{BL_2(n)}  \\ \hline \hline
$n$ & $k^n$ & $N_k(n)$  & $L_k(n)$  && $B_k(n)$  & $BL_k(n)$  \\ \hline \hline
1 & \tabEntry{3}{2} & \tabEntry{3}{2} & \tabEntry{3}{2} && \tabEntry{3}{2} & \tabEntry{3}{2} \\
2 & \tabEntry{9}{4} & \tabEntry{6}{3} & \tabEntry{3}{1} && \tabEntry{6}{3} & \tabEntry{3}{1} \\
3 & \tabEntry{27}{8} & \tabEntry{11}{4} & \tabEntry{8}{2} && \tabEntry{10}{4} & \tabEntry{7}{2} \\
4 & \tabEntry{81}{16} & \tabEntry{24}{6} & \tabEntry{18}{3} && \tabEntry{21}{6} & \tabEntry{15}{3} \\
5 & \tabEntry{243}{32} & \tabEntry{51}{8} & \tabEntry{48}{6} && \tabEntry{39}{8} & \tabEntry{36}{6} \\
6 & \tabEntry{729}{64} & \tabEntry{130}{14} & \tabEntry{116}{9} && \tabEntry{92}{13} & \tabEntry{79}{8} \\
\ldots \\ \hline\hline
\end{tabular}
\end{center}
\caption{Number of equivalence classes of $n$-periodic stationary solutions of \eqref{eq:lde}. In each pair, the former number corresponds to all equivalence classes and the latter number in the parentheses to the equivalence classes formed by asymptotically stable solutions.}\label{tab:number:of:words}
\end{table}

\begin{exmp}
There are 27 distinct 3-periodic stationary solutions of \eqref{eq:lde}. Considering translations $\mathcal{T}_\ell$, they form 11 equivalence classes:
\begin{flalign*}
[u_\mathfrak{0}]_\mathcal{T} &= \{ u_\mathfrak{0} \}, \, 
[u_\mathfrak{a}]_\mathcal{T}=\{ u_\mathfrak{a} \}, \, 
[u_\mathfrak{1}]_\mathcal{T} = \{ u_\mathfrak{1}\}; \,\\
[u_\mathfrak{00a}]_\mathcal{T} &= \{ u_\mathfrak{00a}, u_\mathfrak{0a0}, u_\mathfrak{a00} \}, \, 
[u_\mathfrak{001}]_\mathcal{T} = \{ u_\mathfrak{001}, u_\mathfrak{010}, u_\mathfrak{100} \}, \,
[u_\mathfrak{0aa}]_\mathcal{T} = \{ u_\mathfrak{0aa}, u_\mathfrak{aa0}, u_\mathfrak{a0a} \}, \,
[u_\mathfrak{0a1}]_\mathcal{T} = \{ u_\mathfrak{0a1}, u_\mathfrak{a10}, u_\mathfrak{10a} \}, \,\\ 
[u_\mathfrak{01a}]_\mathcal{T} &= \{ u_\mathfrak{01a}, u_\mathfrak{1a0}, u_\mathfrak{a01} \}, \, 
[u_\mathfrak{011}]_\mathcal{T} = \{ u_\mathfrak{011}, u_\mathfrak{110}, u_\mathfrak{101} \}, \,
[u_\mathfrak{aa1}]_\mathcal{T} = \{ u_\mathfrak{aa1}, u_\mathfrak{a1a}, u_\mathfrak{1aa} \}, \,
[u_\mathfrak{a11}]_\mathcal{T} = \{ u_\mathfrak{a11}, u_\mathfrak{11a}, u_\mathfrak{1a1} \}. 
\end{flalign*}
The former 3 classes correspond to constant solutions and are thus not primitive $3$-periodic solutions, while the remaining 8 classes correspond to primitive stationary $3$-periodic solutions. 
Upon taking  
reflections $\mathcal{R}$ into %the 
consideration, there are only 10 equivalence classes and 7 corresponding to primitive periodic solutions, since $[u_{\mathfrak{0a1}}]_\mathcal{T}$ and $[u_{\mathfrak{01a}}]_\mathcal{T}$ form one equivalence class
\[
[u_{\mathfrak{0a1}}]_\mathcal{TR}= \{u_\mathfrak{0a1}, u_\mathfrak{a10}, u_\mathfrak{10a} ,u_{\mathfrak{1a0}},u_{\mathfrak{01a}},u_{\mathfrak{a01}} \}.
\]
The 8 asymptotically stable solutions form 4 equivalence classes with respect to translations and 2 are primitive -  $[u_{\mathfrak{001}}]_\mathcal{T}$ and $[u_{\mathfrak{011}}]_\mathcal{T}$. These classes are not affected by the reflection $\mathcal{R}$ and can be ordered  as
\[
[u_{\mathfrak{0}}]_\mathcal{T} \triangleleft [u_{\mathfrak{001}}]_\mathcal{T} \triangleleft [u_{\mathfrak{011}}]_\mathcal{T}  \triangleleft [u_{\mathfrak{1}}]_\mathcal{T}.
\]
where $[u_{\gw_A}] \triangleleft [u_{\gw_B}]$ means that the Lyndon representatives satisfy $u_{\gw_A}<u_{\gw_B}$.

The ordering of asymptotically stable $n$-periodic solutions is only partial for periods with $n>3$. For example, Lemma \ref{l:ordering} implies that 6 equivalence classes of $4$-periodic solutions can be partially ordered in the following way.

\begin{center}
\includegraphics[width=0.4\textwidth]{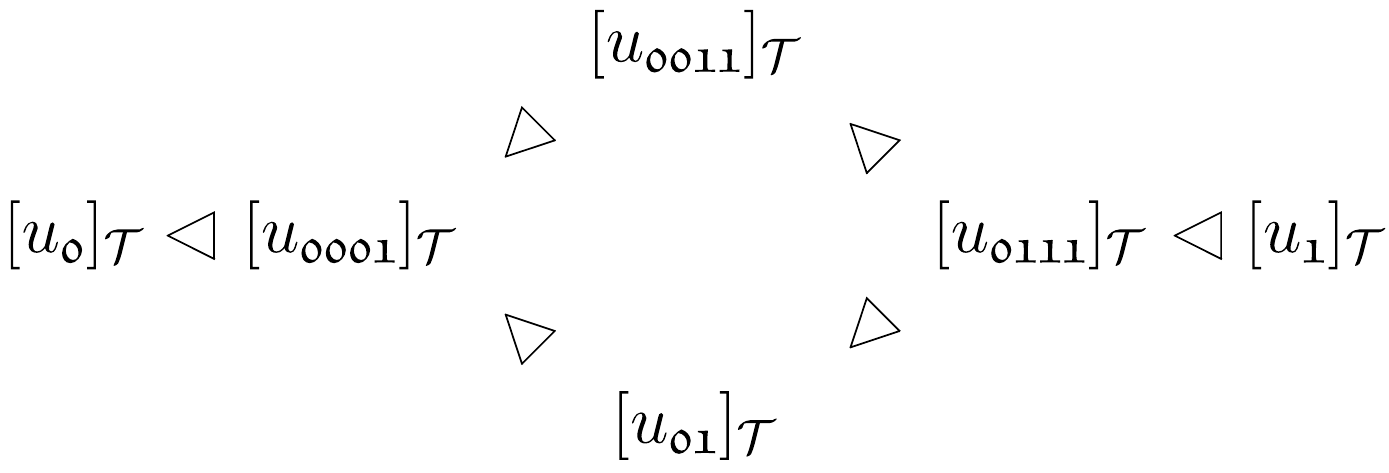}
\end{center}
\noindent The existence regions of the four spatially heterogeneous equivalence classes are depicted in Fig.~\ref{f:regions:4} and the respective Lyndon representatives of five classes corresponding to the ``upper path'' in this diagram are sketched in Fig.~\ref{f:4periodic}.
\end{exmp}

\begin{rmk}\label{r:asymptotics}
Finally, let us note that analyzing the sums in \eqref{eq:NknLkn}-\eqref{eq:BLkn} we arrive to the asymptotic estimates
\[
N_k(n) \sim L_k(n) \sim \dfrac{k^n}{n} \text{, and } B_k(n) \sim BL_k(n) \sim \dfrac{k^n}{2n} \text{ as } n\rightarrow \infty.
\]
\end{rmk}

\noindent\textbf{Acknowledgements.} HJH and LM acknowledge support from the Netherlands Organization for Scientific Research (NWO) (grants 639.032.612, 613.001.304). PS and V\v{S} acknowledge the support of the project LO1506 of the Czech Ministry of Education, Youth and Sports under the program NPU I.

\footnotesize
\bibliographystyle{klunumHJ}

\end{document}